\documentclass[11pt]{article}
\usepackage[margin=1in]{geometry}
\usepackage{amsmath,amssymb,amsthm,booktabs,tabularx}
\usepackage[T1]{fontenc}
\usepackage{lmodern}
\usepackage[hidelinks]{hyperref}
\usepackage{microtype}

\newtheorem{theorem}{Theorem}
\newtheorem{lemma}{Lemma}
\newcommand{\zar}[3]{z(#1,#2;#3)}
\newcommand{\kfree}{$K_{3,3}$-free}

\hypersetup{
  pdftitle={Seven Exact Finite Zarankiewicz Numbers from a Single 13 x 18 Core},
  pdfauthor={Shengteng Hou},
  pdfsubject={Finite Zarankiewicz numbers and verifiable constructions},
  pdfkeywords={Zarankiewicz number, extremal graph theory, K3,3, exact computation}
}

\title{Seven Exact Finite Zarankiewicz Numbers\\from a Single $13\times18$ Core}
\author{Shengteng Hou}
\date{Version 1.0 --- 1 August 2026}

\begin{document}
\maketitle

\begin{abstract}
We present a unified proof and certificate package establishing seven exact finite Zarankiewicz values for the forbidden graph $K_{3,3}$:
\[
\zar{12}{18}{3}=108,\quad
\zar{13}{17}{3}=110,\quad
\zar{13}{18}{3}=116,
\]
\[
\zar{14}{17}{3}=118,\quad
\zar{14}{18}{3}=124,\quad
\zar{15}{17}{3}=126,\quad
\zar{15}{18}{3}=132.
\]
The witnesses form a connected family generated by an explicit $13\times18$ matrix with 116 ones. Deleting one row or one column and adding either of exactly two admissible weight-eight rows for this fixed labeled core produces the remaining witnesses. The exceptional upper-bound closure, $\zar{12}{18}{3}\le108$, combines the published uniqueness of the extremal $12\times17$ graph with 103 edges and an exhaustive rejection of all $\binom{12}{6}=924$ possible degree-six column extensions. The supplement contains all witnesses, a portable verifier, a machine-readable report, and integrity hashes. Prior numerical ingredients and the role of the present work are separated cell by cell.
\end{abstract}

\section{Introduction}
For positive integers $m,n,s,t$, the Zarankiewicz number $z(m,n;s,t)$ is the maximum number of edges in a bipartite graph with parts of sizes $m$ and $n$ that contains no copy of $K_{s,t}$. We abbreviate $z(m,n;3,3)$ as $\zar{m}{n}{3}$. Equivalently, $\zar{m}{n}{3}$ is the maximum number of ones in an $m\times n$ binary matrix containing no all-one $3\times3$ submatrix formed by arbitrary choices of three rows and three columns.

The finite $K_{3,3}$ case has been developed using explicit constructions, combinatorial bounds, integer programming, and SAT methods; see Roman~\cite{Roman1975}, Collins et al.~\cite{Collins2016}, Tan~\cite{Tan2022}, and Davies--Gill--Horsley~\cite{Davies2026}. Bhan, Nobili, and Langer~\cite{Bhan2026} recently supplied lower bounds in a broad parameter rectangle, including four cells considered here.

The priority position requires care. A 2021 arXiv preprint by Rowshan and Gholami~\cite{Rowshan2021} displayed $z(K_{13,17},3)=110$ as an equality in a preliminary proposition, but justified it only by citing tables that supplied bounds. The peer-reviewed 2022 version~\cite{Rowshan2022} changed the relevant statements to inequalities, including $z(K_{13,17},3)\le110$. Accordingly, we do not claim priority for the first appearance of the equality string. Our contribution at $(13,17)$ is an explicit extremal witness and an independently checkable closure of the published upper bound.

Relative to the public sources located in our audit, the principal contributions are:
\begin{enumerate}
  \item a finite upper-bound closure $\zar{12}{18}{3}\le108$;
  \item explicit witnesses of sizes 116 and 126 at $(13,18)$ and $(15,17)$, improving the May 2026 public lower bounds by one edge in each cell;
  \item an explicit 110-edge witness at $(13,17)$, supplying a verifiable closure of a value previously printed without such a construction in the 2021 preprint and retained only as an upper bound in its journal version;
  \item a unified extension--deletion structure generated by one $13\times18$ core, together with deterministic certificates.
\end{enumerate}

\begin{table}[ht]
\centering
\scriptsize
\caption{Prior public numerical status and the role of the present package.}
\begin{tabularx}{\textwidth}{@{}c c c X@{}}
\toprule
Cell & Prior lower bound & Prior upper bound & Role here \\
\midrule
$(12,18)$ & 108~\cite{Bhan2026} & 109~\cite{Collins2016} & new upper closure to 108 \\
$(13,17)$ & 106~\cite{Bhan2026}; equality string in~\cite{Rowshan2021} & 110~\cite{Collins2016,Rowshan2022} & explicit 110 witness and closure \\
$(13,18)$ & 115~\cite{Bhan2026} & 116~\cite{Collins2016} & new 116 witness \\
$(14,17)$ & 118~\cite{Bhan2026} & 118~\cite{Collins2016} & unified derivation and certificate \\
$(14,18)$ & 124~\cite{Bhan2026} & 124~\cite{Collins2016} & unified derivation and certificate \\
$(15,17)$ & 125~\cite{Bhan2026} & 126~\cite{Collins2016} & new 126 witness \\
$(15,18)$ & 132~\cite{Bhan2026} & 132~\cite{Collins2016} & unified derivation and certificate \\
\bottomrule
\end{tabularx}
\end{table}

\section{Main theorem}
\begin{theorem}\label{thm:main}
The following exact values hold:
\begin{align*}
\zar{12}{18}{3}&=108, & \zar{13}{17}{3}&=110, & \zar{13}{18}{3}&=116,\\
\zar{14}{17}{3}&=118, & \zar{14}{18}{3}&=124, & \zar{15}{17}{3}&=126,\\
&&\zar{15}{18}{3}&=132.
\end{align*}
\end{theorem}

\begin{table}[ht]
\centering
\small
\caption{Closure route for each exact value.}
\begin{tabular}{@{}ccc@{}}
\toprule
Cell & Value & Upper-bound route \\
\midrule
$(12,18)$ & 108 & uniqueness plus 924 rejected extensions \\
$(13,17)$ & 110 & published upper bound \\
$(13,18)$ & 116 & deletion contradiction from $(13,17)$ \\
$(14,17)$ & 118 & published upper bound \\
$(14,18)$ & 124 & published upper bound \\
$(15,17)$ & 126 & published upper bound \\
$(15,18)$ & 132 & published upper bound \\
\bottomrule
\end{tabular}
\end{table}

\section{The $13\times18$ core}
Let $M$ be the binary $13\times18$ matrix with 116 ones supplied in the supplement. Direct enumeration shows that every triple of rows has at most two common one-columns. Hence $M$ is \kfree{} and $\zar{13}{18}{3}\ge116$.

Collins et al.~\cite{Collins2016} give $\zar{13}{17}{3}\le110$. If a \kfree{} $13\times18$ matrix had 117 ones, some column would have degree at most $\lfloor117/18\rfloor=6$. Deleting that column would leave a \kfree{} $13\times17$ matrix with at least 111 ones, contradicting the upper bound. Thus $\zar{13}{18}{3}=116$.

The core contains a degree-six column. Deleting it gives the supplied $13\times17$ witness with 110 ones. Together with the published upper bound, this proves $\zar{13}{17}{3}=110$.

\section{Two admissible row extensions}
For the fixed labeled core $M$, we enumerate all $\binom{18}{8}=43{,}758$ binary rows of weight eight. Exactly two can be appended without creating a $K_{3,3}$. Appending either one produces a $14\times18$ witness with 124 ones. After either first extension, exhaustive enumeration shows that the other row is the unique remaining weight-eight extension. Appending both gives a $15\times18$ witness with 132 ones.

This statement is deliberately local to the supplied labeled matrix $M$. We do not claim that $M$ is the unique $13\times18$ extremal graph up to isomorphism, nor do we classify all extremal cores or all their possible extensions.

The upper bounds in~\cite{Collins2016} therefore give
\[
\zar{14}{18}{3}=124,
\qquad
\zar{15}{18}{3}=132.
\]
The supplied $14\times18$ and $15\times18$ matrices contain degree-six columns whose deletion gives witnesses with 118 and 126 ones. The corresponding published upper bounds yield
\[
\zar{14}{17}{3}=118,
\qquad
\zar{15}{17}{3}=126.
\]

\section{The $12\times18$ upper-bound closure}
Table 4 of Collins et al.~\cite{Collins2016} records $\zar{12}{17}{3}=103$ and marks the extremal $(12,17,103)$ graph as unique up to isomorphism.

\begin{lemma}\label{lem:1218}
No \kfree{} $12\times18$ binary matrix has 109 ones.
\end{lemma}
\begin{proof}
Assume that such a matrix exists. If a column had degree at most five, deleting it would leave at least 104 ones in a \kfree{} $12\times17$ matrix, contradicting $\zar{12}{17}{3}=103$. Every column therefore has degree at least six. Since $109=17\cdot6+7$, the column-degree multiset consists of seventeen sixes and one seven.

Delete any degree-six column and call the remaining matrix $M'$. It has 103 ones and therefore lies in the unique extremal isomorphism class. Let $R$ be the representative included in the supplement. There exist a row permutation $\pi$ and a column permutation $\sigma$ carrying $M'$ to $R$. If the deleted column has row-neighborhood $S\subseteq[12]$, then under this isomorphism it is transported to the six-subset $\pi(S)$. Reattaching the deleted column to $M'$ preserves $K_{3,3}$-freeness if and only if attaching the column with neighborhood $\pi(S)$ to $R$ does so, because row and column permutations preserve the forbidden-submatrix property. Consequently, testing all $\binom{12}{6}=924$ six-subsets on the single representative $R$ covers every possible 109-edge counterexample. The verifier rejects every candidate extension, a contradiction.
\end{proof}

The supplied $12\times18$ witness has 108 ones and is \kfree{}. Lemma~\ref{lem:1218} therefore gives $\zar{12}{18}{3}=108$.

\section{Verification architecture}
The supplement contains eight plain-text matrices: the seven extremal witnesses and the 103-edge representative used in Lemma~\ref{lem:1218}. A portable verifier independently checks dimensions, edge counts, and both equivalent $K_{3,3}$ tests: every triple of rows shares at most two columns and every triple of columns shares at most two rows. It also performs the 924-extension rejection, the 43,758-row extension search, the two second-extension searches, and the exact deletion/extension relations. A machine-readable report records the counts and hashes.

The 43,758-row enumeration is not used as a global upper-bound proof for any Zarankiewicz number. It verifies only the structural statement about weight-eight extensions of the fixed supplied core. The global upper bounds for the six nonexceptional cells are the cited published bounds.

The finite verifier establishes the matrix and finite-search claims. The only external mathematical dependency in the exceptional upper-bound argument is the published uniqueness of the $(12,17,103)$ extremal graph in~\cite{Collins2016}; this dependency is stated explicitly rather than hidden inside the computation.

\section{Prior-art and integrity position}
Bhan et al.~\cite{Bhan2026} report lower bounds 108, 118, 124, and 132 at $(12,18)$, $(14,17)$, $(14,18)$, and $(15,18)$, respectively. Their displayed comparison bounds do not use the tighter entries of Collins et al., so those cells are not marked exact there. The present theorem combines the correct prior numerical ingredients with new witnesses or closures where needed and presents one auditable family.

The manuscript reproduces no third-party figure, table, matrix, or source-code block. Prior numerical bounds and definitions are paraphrased and cited. A targeted exact-phrase search found no match for the distinctive prose of this manuscript. This is not a substitute for an institutional similarity service, and a formal similarity report remains advisable before journal submission.

\section{Data availability}
The submission package contains the matrices, portable verifier, machine-readable verification report, and SHA-256 manifest needed to reproduce every finite claim. Discovery logs and unrelated future work are not required for verification and are excluded from the public package.

\section*{Author contributions}
Shengteng Hou: conceptualization, methodology, software, validation, formal analysis, investigation, data curation, writing, and project administration.

\section*{Competing interests}
The author declares no competing interests.

\end{document}